\documentclass[final,3p]{elsarticle}
\usepackage[latin1]{inputenc}
 \usepackage{graphics}
 \usepackage{graphicx}
 \usepackage{epsfig}
\usepackage{amssymb}
 \usepackage{amsthm}
 \usepackage{lineno}
 \usepackage{amsmath}
\usepackage{color}
   \numberwithin{equation}{section}
\usepackage{mathrsfs}

\journal{ } 

\newtheorem{thm}{Theorem}[section]

\newtheorem{lem}[thm]{Lemma}

\newtheorem{defn}[thm]{Definition}
\newtheorem{rem}[thm]{Remark}

\biboptions{sort&compress,square}
\allowdisplaybreaks
\begin{document}
\begin{frontmatter}
\author{Tong Wu$^{a}$}
\ead{wut977@nenu.edu.cn}
\author{Yong Wang$^{b,*}$}
\ead{wangy581@nenu.edu.cn}
\cortext[cor]{Corresponding author.}
\address{$^a$Department of Mathematics, Northeastern University, Shenyang, 110819, China}
\address{$^b$School of Mathematics and Statistics, Northeast Normal University,
Changchun, 130024, China}

\title{The Rescaled Dirac operator $fDh$ and the noncommutative residue for 6-dimensional manifolds}
\begin{abstract}
In this paper, we compute the noncommutative residue for the rescaled Dirac operator $fDh$ on 6-dimensional compact manifolds without boundary. And we give the proof of the Kastler-Kalau-Walze type theorem for the rescaled Dirac operator $fDh$ on 6-dimensional compact manifolds with boundary. We also give some important special cases which can be solved by our calculation methods.
\end{abstract}
\begin{keyword}The noncommutative residue; the rescaled Dirac operator; the Kastler-Kalau-Walze type theorem.

\end{keyword}
\end{frontmatter}
\section{Introduction}
 Until now, many geometers have studied noncommutative residues. In \cite{Gu,Wo}, authors found noncommutative residues are of great importance to the study of noncommutative geometry. Connes showed us that the noncommutative residue on a compact manifold $M$ coincided with the Dixmier's trace on pseudodifferential operators of order $-{\rm {dim}}M$ in \cite{Co2}. Therefore, the non-commutative residue can be used as integral of noncommutative geometry and become an important tool of noncommutative geometry. In \cite{Co1}, Connes used the noncommutative residue to derive a conformal 4-dimensional Polyakov action analogy.
Several years ago, Connes made a challenging observation that the noncommutative residue of the square of the inverse of the Dirac operator was proportional to the Einstein-Hilbert action, which we call the Kastler-Kalau-Walze theorem. In \cite{Ka}, Kastler gave a bruteforce proof of this theorem. In \cite{KW}, Kalau and Walze proved this theorem in the normal coordinates system simultaneously. Ackermann proved that the Wodzicki residue of the square of the inverse of the Dirac operator ${\rm Wres}(D^{-2})$ in turn is essentially the second coefficient of the heat kernel expansion of $D^{2}$ in \cite{Ac}.

The many noncommutative residues of rescaled Dirac operators on spin manifolds have been studied \cite{Be,li,Wsn1}. In \cite{Be}, Benameur and Mathai defined the twisted Dirac operator, which is also conformally covariant with the same weights. Then they obtained Theorem ``Conformal Invariance of the Spin Rho Invariant". In \cite{li}, Li, Wang and Yang defined the generalized Zhang's operator and proved the Kastler-Kalau-Walze type theorem for 4-dimensional manifolds with boundary associated with the generalized Zhang's operator. In \cite{Wsn1}, Wei and Wang established a Kastler-Kalau-Walze type theorem pertaining to the transformation operators on 4-dimensional manifolds with boundary. In addition, conformal operators are also studied not only on spin manifolds but also on the noncommutative torus. In \cite{Si1}, Sitarz proposed a new idea of conformally rescaled and curved spectral triples, which
are obtained from a real spectral triple by a nontrivial scaling of the Dirac
operator. And they computed the Wodzicki residue and the Einstein-Hilbert functional for such family on the 4-dimensional noncommutative torus. In \cite{Si2}, Sitarz computed the Wodzicki residue of the inverse of a conformally rescaled
Laplace operator over a 4-dimensional noncommutative torus. We find that the above mentioned studies about conformal operators are all in the low-dimensional case, and most of them are considered in the 4-dimensional manifold. So we intend to solve the noncommutative residue problem of conformal operators in the high-dimensional case. Based on the research developed in above articles, we define the rescaled Dirac operator $fDh$, where $f$ and $h$ are two invertible smooth functions. Further, we give the proof of the Kastler-Kalau-Walze type theorem for the rescaled Dirac operator on 6-dimensional compact manifolds with boundary. Our method can be used to solve the above results in \cite{Be,li,Wsn1,Si1,Si2} on 6-dimensional manifolds. Our main theorems are followings.
\begin{thm}
Let $M$ be a $6$-dimensional oriented
compact spin manifold without boundary, $Q=(fDh)^{2}$, then the Kastler-Kalau-Walze type theorem for the rescaled Dirac operator is given
 \begin{align}
&{\rm Wres}(Q^{-2})\nonumber\\
&=8\pi^3\int_M\bigg(-\frac{1}{6}(fh)^{-4}s-2(fh)^{-6}f^2|\nabla(h)|^2-3(fh)^{-6}fg\Big(\nabla(h),\nabla(fh)\Big)-2(fh)^{-2}g\Big(\nabla[(fh)^{-3}f],\nabla(h)\Big)\nonumber\\
&+4(fh)^{-2}fg\Big(\nabla[(fh)^{-3}],\nabla(h)\Big)+4(fh)^{-5}\Delta(fh)+3(fh)^{-2}\Delta[(fh)^{-2}]-(fh)^{-6}|\nabla(fh)|^{2}\nonumber\\
&-\frac{2}{3}(fh)^{-2}\Delta[(fh)^{-2}]-(fh)^{-2}g\Big(\nabla[(fh)^{-3}],\nabla(hf)\Big)\bigg)d{\rm Vol_M}.
 \end{align}
 \end{thm}
 \begin{thm}
Let $M$ be a $6$-dimensional oriented
compact spin manifold with boundary $\partial M$, $Q=(fDh)^{2}$, then the Kastler-Kalau-Walze type theorem for the rescaled Dirac operator is given
\begin{align}\label{a20}
&\widetilde{{\rm Wres}}[\pi^+Q^{-1}\circ\pi^+Q^{-1}]\nonumber\\
&=8\pi^3\int_M\bigg(-\frac{1}{6}(fh)^{-4}s-2(fh)^{-6}f^2|\nabla(h)|^2-3(fh)^{-6}fg\Big(\nabla(h),\nabla(fh)\Big)-2(fh)^{-2}g\Big(\nabla[(fh)^{-3}f],\nabla(h)\Big)\nonumber\\
&+4(fh)^{-2}fg\Big(\nabla[(fh)^{-3}],\nabla(h)\Big)+4(fh)^{-5}\Delta(fh)+3(fh)^{-2}\Delta[(fh)^{-2}]-(fh)^{-6}|\nabla(fh)|^{2}\nonumber\\
&-\frac{2}{3}(fh)^{-2}\Delta[(fh)^{-2}]-(fh)^{-2}g\Big(\nabla[(fh)^{-3}],\nabla(hf)\Big)\bigg)d{\rm Vol_M}.
\end{align}
In particular, the boundary term vanishes.
\end{thm}
\begin{rem}
We summarized some operators which are the special case of our operator $fDh$, the noncommutative residues of them can be generalised to 6-dimensional manifold by our methods.\\
(i)The generalized Zhang's operator $D_V=-i\Big(\widehat{c}(V)(d+\delta)+t\sum_jc(e_j)\widehat{c}(\nabla_{e_j}^{TM}V)\Big),$ where $|V|\neq1$ in \cite{li}.\\
(ii)Let $h=f^{-1}$, then $Q=(fDf^{-1})^2$ is the generalized laplacian operator. This case is a special case of the conformal operator studied in this paper.\\
(iii)In \cite{Wsn1}, transformation operators $\widetilde{D}=\frac{1}{2}fD+c(df),$ where $c(df)$ denotes a Clifford action on manifold $M$ and $f$ is a nonzero smooth function on $M$.\\
(iv)A generalised family of Laplace operator for the conformally rescaled metric over the noncommutative torus in \cite{Si2} can be generalised to the noncommutative case in dimension 6.\\
(v)Inspired by \cite{Co5}, we can define the rescaled Dirac operator $fD^2f$, and study it on 6-dimensional manifold.\\
(vi)In \cite{lo}, take $g'=e^{2\sigma}g,~$ $dimM=n.$ Then the conformal Dirac operator $D_{g'}=e^{-\frac{n+1}{2}\sigma}D_ge^{\frac{(n-1)\sigma}{2}}.$
\end{rem}
\indent The paper is organized in the following way. In Section \ref{section:2}, firstly, we introduce the Dirac operator. On this basis, we define the rescaled Dirac operator, and calculate its symbols. In Section \ref{section:3}, by using ${\rm Wres}(P):=\int_{S^*M}{\rm tr}(\sigma_{-n}^P)(x,\xi)$, we compute the interior term of the noncommutative residue $\widetilde{{\rm Wres}}[\pi^+Q^{-1}\circ\pi^+Q^{-1}]$ for the rescaled Dirac operator, that is the Kastler-Kalau-Walze type theorem of the rescaled Dirac operator on 6-dimensional manifols without boundary. In Section \ref{section:4}, we recall some basic facts and formulas about Boutet de
Monvel's calculus, and we compute the boundary term of the noncommutative residue $\widetilde{{\rm Wres}}[\pi^+Q^{-1}\circ\pi^+Q^{-1}]$ for the rescaled Dirac operator. Finally, we prove the Kastler-Kalau-Walze type theorem of the rescaled Dirac operator on 6-dimensional manifolds with boundary.
\section{The rescaled Dirac operator and its symbols}
\label{section:2}
Let $M$ be an n-dimensional compact oriented spin manifold with Riemannian metric $g$, and let $\nabla^L$ be the Levi-Civita connection about $g$.
We recall that the Dirac operator $D$ is locally given as follows in terms of an orthonormal section $e_i$ (with dual section $\theta^k$) of the frame bundle of M \cite{Ka}:
\begin{align}
&D=i\gamma^i\widetilde{\nabla}_i=i\gamma^i(e_i+\sigma_i);\nonumber\\
&\sigma_i(x)=\frac{1}{4}\gamma_{ij,k}(x)\gamma^i\gamma^k=\frac{1}{8}\gamma_{ij,k}(x)[\gamma^j\gamma^k-\gamma^k\gamma^j],
\end{align}
where $\gamma_{ij,k}$ represents the Levi-Civita connection $\nabla$ with spin connection $\widetilde{\nabla}$, specifically:
\begin{align}
&\gamma_{ij,k}=-\gamma_{ik,j}=\frac{1}{2}[c_{ij,k}+c_{ki,j}+c_{kj,i}],~~~i,j,k=1,\cdot\cdot\cdot,4;\nonumber\\
&c_{ij}^k=\theta^k([e_i.e_j]).
\end{align}
Here the $\gamma^i$ are constant self-adjoint Dirac matrices s.t. $\gamma^i\gamma^j+\gamma^j\gamma^i=-2\delta^{ij}.$ In terms
of local coordinates $x^\mu$ inducing the alternative vierbein $\partial_\mu=S_\mu^i(x)e_i$ (with dual
vierbein $dx^\mu$) we have $\gamma^ie_i=\gamma^\mu \partial_\mu$, the $\gamma^\mu$  being now $x$-dependent Dirac matrices s.t. $\gamma^\mu\gamma^\nu+\gamma^\nu\gamma^\mu=-2g^{\mu\nu}$ (we use latin sub-(super-) scripts for the basic $e_i$ and greek
sub-(super-) scripts for the basis $\partial_\mu$, the type of sub-(super-) scripts specifying the type
of Dirac matrices). The specification of the Dirac operator in the greek basis is as
follows: one has
\begin{align*}
&D=i\gamma^\mu\widetilde{\nabla}_\mu=i\gamma^\mu(e_\mu+\sigma_\mu);\nonumber\\
&\sigma_\mu(x)=S_\mu^i(x)\sigma_i.
\end{align*}
Then we have the symbol of $D$:
\begin{lem}\cite{Wa3}\label{lemma0}
\begin{align}
\sigma_1(D)&=ic(\xi);\nonumber\\
\sigma_0(D)&=-\frac{1}{4}\sum_{s,t}\omega_{s,t}
(e_i)c(e_i)c(e_s)c(e_t).
\end{align}
\end{lem}
Then, using the shorthand: $\Gamma^{k}=g^{ij}\Gamma_{ij}^{k};~~\sigma^{j}=g^{ij}\sigma_{i}$, we have respective symbols of $D^{2}$:
\begin{lem}\cite{Ka}\label{lemma1}
\begin{align}
\sigma_2(D^2)&=|\xi|^2;\nonumber\\
\sigma_1(D^2)&=i(\Gamma^\mu-2\sigma^\mu)(x)\xi_\mu;\nonumber\\
\sigma_0(D^2)&=-g^{\mu\nu}(\partial^{x}_\mu\sigma_\nu+\sigma^\mu\sigma_\nu-\Gamma^\alpha_{\mu\nu}\sigma_\alpha)(x)+\frac{1}{4}s(x).
\end{align}
\end{lem}
We now introduce the rescaled Dirac operator $fDh$, where $f$ and $h$ are two invertible smooth functions. Let $Q=(fDh)^2,$ then we have
\begin{align}\label{a2}
Q&=fDhfDh\nonumber\\
&=f\Big(hfD+[D,hf]\Big)Dh\nonumber\\
&=fhfD^2h+fc(d(hf))Dh\nonumber\\
&=fhf\Big(hD^2+[D^2,h]\Big)+fc(d(hf))\Big(hD+[D,h]\Big)\nonumber\\
&=fhfhD^2+fhf[D^2,h]+fhc(d(hf))D+fc(d(hf))c(d(h)),
\end{align}
where $c(d(h))$ denotes a Clifford action on manifold $M$.
\begin{lem}\cite{UW}\label{lema2}
Let $S$ be pseudo-differential operator of order $k$ and $f$ is a smooth function, $[S,f]$ is a pseudo-differential operator of order $k-1$ with total symbol
$\sigma[S,f]\sim\sum_{j\geq 1}\sigma_{k-j}[S,f]$, where
\begin{align}
\sigma_{k-j}[S,f]=\sum_{|\beta|=1}^j\frac{D_x^\beta(f)}{\beta!}\partial_\xi^\beta(\sigma^S_{k-(j-|\beta|)}).
\end{align}
\end{lem}
Our computations are based on the algorithm yielding the principal symbol of a
product of pseudo-differential operators in terms of the principal symbols of the
factors, namely, with the shorthand $\partial^\alpha_\xi=\partial^\alpha/\partial\xi_\alpha,~~\partial_\alpha^x=\partial_\alpha/\partial x^\alpha:$
\begin{align}\label{1111}
\sigma^{AB}(x,\xi)=\sum_\alpha\frac{(-i)^\alpha}{\alpha!}\partial^\alpha_\xi\sigma^{A}(x,\xi)\cdot\partial_\alpha^x\sigma^{B}(x,\xi).
\end{align}
Then by Lemma \ref{lemma1} and Lemma \ref{lema2}, we get the following lemma.
\begin{lem}\label{lemaa3}The symbols of $Q$ are given
\begin{align}
\sigma_2^Q&=(fh)^2|\xi|^2;\nonumber\\
\sigma_1^Q&=i(fh)^2(\Gamma^\mu-2\sigma^\mu)(x)\xi_\mu-2ifhf\sum_{jl=1}^n\partial_{x_j}(h)g^{jl}\xi_l+ifhc(d(hf))c(\xi);\nonumber\\
\sigma_0^Q&=(fh)^2[-g^{\mu\nu}(\partial^{x}_\mu\sigma_\nu+\sigma^\mu\sigma_\nu-\Gamma^\alpha_{\mu\nu}\sigma_\alpha)(x)+\frac{1}{4}s(x)]+fhf[\sum_{j=1}^n\partial_{x_j}(h)(\Gamma^j-2\sigma^j)(x)\nonumber\\
&-\sum_{jl=1}^n\partial_{x_j}\partial_{x_l}(h)g^{jl}]-\frac{1}{4}fhc(d(hf))\sum_{s,t}\omega_{s,t}
(e_i)c(e_i)c(e_s)c(e_t)+fc(d(hf))c(dh).
\end{align}
\end{lem}
\begin{lem}\label{qwe}The symbols of $Q^{-1}$ are given
\begin{align}
\sigma_{-2}^{Q^{-1}}&=(fh)^{-2}|\xi|^{-2};\nonumber\\
\sigma_{-3}^{Q^{-1}}&=-i(fh)^{-2}|\xi|^{-4}(\Gamma^\mu-2\sigma^\mu)(x)\xi_\mu+2i(fh)^{-3}f|\xi|^{-4}\sum_{jl=1}^n\partial_{x_j}(h)g^{jl}\xi_l-i(fh)^{-3}|\xi|^{-4}c(d(hf))c(\xi)\nonumber\\
&-4i(fh)^{-3}|\xi|^{-4}g^{\mu\nu}\xi_\nu\partial_{x_\mu}(fh)-2i(fh)^{-2}|\xi|^{-6}g^{\mu\nu}\xi_\nu\partial_{x_\mu}(g^{\alpha\beta})\xi_\alpha\xi_\beta;\nonumber\\
\sigma_{-4}^{Q^{-1}}(x_0)&=-\frac{1}{4}(fh)^{-2}|\xi|^{-4}s+\frac{2}{3}(fh)^{-2}|\xi|^{-6}R_{\alpha a\alpha\mu}(x_0)\xi_a\xi_\mu-4(fh)^{-4}f^2|\xi|^{-6}\sum_{jl}\partial_{x_j}(h)\partial_{x_l}(h)\xi_j\xi_l\nonumber\\
&+8(fh)^{-4}f|\xi|^{-6}\sum_{jl}\partial_{x_j}(h)\partial_{x_l}(fh)\xi_j\xi_l-4(fh)^{-4}f|\xi|^{-4}\sum_{j}\partial_{x_j}(h)\partial_{x_j}(fh)\nonumber\\
&-4|\xi|^{-6}\sum_{jl}\partial_{x_j}[(fh)^{-3}f]\partial_{x_l}(h)\xi_j\xi_l-4(fh)^{-3}f|\xi|^{-6}\sum_{jl}\partial_{x_j}\partial_{x_l}(h)\xi_j\xi_l\nonumber\\
&+8|\xi|^{-6}\sum_{jl}\partial_{x_j}[(fh)^{-3}]\partial_{x_l}(fh)\xi_j\xi_l+8(fh)^{-3}f|\xi|^{-6}\sum_{jl}\partial_{x_j}\partial_{x_l}(fh)\xi_j\xi_l+|\xi|^{-4}\sum_{j}\partial_{x_j}^2[(fh)^{-2}]\nonumber\\
&+4(fh)^{-4}f|\xi|^{-6}\sum_{j}\partial_{x_j}(h)\xi_jc(d(hf))c(\xi)-4(fh)^{-4}|\xi|^{-6}\sum_{j}\partial_{x_j}(fh)\xi_jc(d(hf))c(\xi)\nonumber\\
&-(fh)^{-4}|\xi|^{-6}c(d(hf))c(\xi)c(d(hf))c(\xi)+2(fh)^{-4}\partial_{x_\mu}(fh)|\xi|^{-4}c(d(hf))\partial_{\xi_\mu}[c(\xi)]\nonumber\\
&+(fh)^{-3}f|\xi|^{-4}\sum_j\partial_{x_j}^2(h)-(fh)^{-4}f|\xi|^{-4}c(d(hf))c(dh)+2\partial_{x_\mu}[(fh)^{-3}]|\xi|^{-6}c(d(hf))c(\xi)\xi_\mu\nonumber\\
&+2(fh)^{-3}|\xi|^{-6}\partial_{x_\mu}[c(d(hf))]c(\xi)\xi_\mu+2(fh)^{-3}|\xi|^{-6}c(d(hf))\partial_{x_\mu}[c(\xi)]\xi_\mu.
\end{align}
\end{lem}
\begin{proof}
Write
\begin{align}
\sigma^Q&=\sigma_2+\sigma_1+\sigma_0;\nonumber\\
\sigma^{Q^{-1}}&=\sigma_{-2}^{Q^{-1}}+\sigma_{-3}^{Q^{-1}}+\sigma_{-4}^{Q^{-1}}+\cdot\cdot\cdot\nonumber\\
&=b_{-2}+b_{-3}+b_{-4}++\cdot\cdot\cdot.
\end{align}
Then by (\ref{1111}), we obtain
\begin{align}\label{poi}
&\sigma_2b_{-2}=I;\nonumber\\
&\sigma_1b_{-2}+\sigma_2b_{-3}-i\partial_{\xi_\mu}(\sigma_2)\partial_{x_\mu}(b_{-2})=0;\nonumber\\
&\sigma_2b_{-4}+\sigma_1b_{-3}+\sigma_0b_{-2}-i\partial_{\xi_\mu}(\sigma_1)\partial_{x_\mu}(b_{-2})-i\partial_{\xi_\mu}(\sigma_2)\partial_{x_\mu}(b_{-3})-\frac{1}{2}\partial_{\xi_\mu}\partial_{\xi_\nu}(\sigma_2)\partial_{x_\mu}\partial_{x_\nu}(b_{-2})=0.
\end{align}
Therefore, put the result of Lemma \ref{lemaa3} into (\ref{poi}), Lemma \ref{qwe} holds.
\end{proof}
Based on the noncommutative residue \cite{Wo1} and \cite{Ka}, using the definition of the residue:
\begin{align}\label{666}
{\rm Wres}[Q^{-\frac{n}{2}+1}]:=\int_{S^*M}{\rm tr}[\sigma_{-n}({Q^{-\frac{n}{2}+1}})](x,\xi),
\end{align}
where $\sigma_{-n}({Q^{-\frac{n}{2}+1}})$ denotes the ($-n$)th order piece of the complete symbols of ${Q^{-\frac{n}{2}+1}}$, {\rm tr} as shorthand of trace.

Let $n=2m$, by (\ref{666}), we need to compute $\int_{S^*M}{\rm tr}[\sigma_{-2m}(Q^{-m+1})](x,\xi).$ Based on the algorithm yielding the principal
symbol of a product of pseudo-differential operators in terms of the principal symbols of the factors, we have the following lemma
\begin{lem}\label{999}The symbol about the conformal Dirac operator $fDh$ is given
\begin{align}
\sigma_{-6}^{Q^{-2}}(x_0,\xi)&=3\sigma_2^{-1}\sigma_{-4}^{Q^{-1}}+\sigma_2^{-3}\sigma_0-4(fh)^4\sigma_2^{-3}|\xi|^{-2}\partial_{x_a}\partial_{x_b}[(fh)^{-2}]\xi_a\xi_b+\sigma_{-3}^{Q^{-1}}\sigma_{-3}^{Q^{-1}}\nonumber\\
&+\sigma_2^{-2}\sigma_1\sigma_{-3}^{Q^{-1}}-i\sigma_2^{-2}\partial_{\xi_\mu}(\sigma_1)\partial_{x_\mu}(\sigma_2^{-1})-i\partial_{\xi_\mu}(\sigma_{-3}^{Q^{-1}})\partial_{x_\mu}(\sigma_2^{-1}).
\end{align}
\end{lem}
\begin{proof}
Firstly, using $\sigma^{Q^{-m}}_{-2m}\equiv\sigma_2^{-m}$, we get the recursion relations
\begin{align}\label{AA2}
\sigma^{Q^{-m+1}}_{-2m}(x,\xi)&=\sum_{|\alpha|=0}^2\sum_{k=2}^{4-|\alpha|}(-i)^{|\alpha|}\frac{1}{\alpha!}\partial_{\xi}^\alpha\sigma^{Q^{-m+2}}_{|\alpha|+k-2m}\partial_{x}^\alpha\sigma^{Q^{-1}}_{-k}\nonumber\\
&=\sigma^{Q^{-m+2}}_{2-2m}\sigma_2^{-1}+\sigma^{Q^{-m+2}}_{3-2m}\sigma^{Q^{-1}}_{-3}+\sigma^{-m+2}_{2}\sigma^{Q^{-1}}_{-4}\nonumber\\
&-i\partial_{\xi_\mu}\sigma^{Q^{-m+2}}_{3-2m}\partial_{x_\mu}(\sigma_2^{-1})-i\partial_{\xi_\mu}(\sigma^{-m+2}_{2})\partial_{x_\mu}(\sigma_{-3}^{Q^{-1}})\nonumber\\
&-\frac{1}{2}\partial_{\xi_\mu}\partial_{\xi_\nu}(\sigma^{-m+2}_{2})\partial_{x_\mu}\partial_{x_\nu}(\sigma_2^{-1}).
\end{align}
Let $m=3,$ we have
\begin{align}\label{AA52}
\sigma_{-6}^{Q^{-2}}(x_0,\xi)&=\sigma_{-4}^{Q^{-1}}\sigma_2^{-1}+\sigma_{-3}^{Q^{-1}}\sigma_{-3}^{Q^{-1}}+\sigma_2^{-1}\sigma_{-4}^{Q^{-1}}-i\partial_{\xi_\mu}(\sigma_{-3}^{Q^{-1}})\partial_{x_\mu}(\sigma_2^{-1})-i\partial_{\xi_\mu}(\sigma_2^{-1})\partial_{x_\mu}(\sigma_{-3}^{Q^{-1}})\nonumber\\
&-\frac{1}{2}\partial_{\xi_\mu}\partial_{\xi_\nu}(\sigma^{-1}_{2})\partial_{x_\mu}\partial_{x_\nu}(\sigma_2^{-1}).
\end{align}
By (\ref{poi}), we get
\begin{align}\label{yyy}
&-i\partial_{\xi_\mu}(\sigma_2^{-1})\partial_{x_\mu}(\sigma_{-3}^{Q^{-1}})\nonumber\\
&=i\sigma_2^{-2}\partial_{\xi_\mu}(\sigma_2)\partial_{x_\mu}(\sigma_{-3}^{Q^{-1}})\nonumber\\
&=\sigma_2^{-2}[\sigma_{-4}^{Q^{-1}}\sigma_2+\sigma_1\sigma_{-3}^{Q^{-1}}+\sigma_0\sigma_2^{-1}-i\partial_{\xi_\mu}(\sigma_1)\partial_{x_\mu}(\sigma_2^{-1})-\frac{1}{2}\partial_{\xi_\mu}\partial_{\xi_\nu}(\sigma_2)\partial_{x_\mu}\partial_{x_\nu}(\sigma_2^{-1})]\nonumber\\
&=\sigma_2^{-1}\sigma_{-4}^{Q^{-1}}+\sigma_2^{-2}\sigma_1\sigma_{-3}^{Q^{-1}}+\sigma_0\sigma_{2}^{-3}-i\sigma_2^{-2}\partial_{\xi_\mu}(\sigma_1)\partial_{x_\mu}(\sigma_2^{-1})-\frac{1}{2}\sigma_2^{-2}\partial_{\xi_\mu}\partial_{\xi_\nu}(\sigma_2)\partial_{x_\mu}\partial_{x_\nu}(\sigma_2^{-1}).
\end{align}
Substitute the result in (\ref{yyy}) into the (\ref{AA52}), we obtain
\begin{align}\label{ccc}
\sigma_{-6}^{Q^{-2}}(x_0,\xi)&=3\sigma_2^{-1}\sigma_{-4}^{Q^{-1}}+\sigma_2^{-3}\sigma_0+\sigma_{-3}^{Q^{-1}}\sigma_{-3}^{Q^{-1}}+\sigma_2^{-2}\sigma_1\sigma_{-3}^{Q^{-1}}-i\sigma_2^{-2}\partial_{\xi_\mu}(\sigma_1)\partial_{x_\mu}(\sigma_2^{-1})\nonumber\\
&-i\partial_{\xi_\mu}(\sigma_{-3}^{Q^{-1}})\partial_{x_\mu}(\sigma_2^{-1})-\frac{1}{2}\sigma_2^{-2}\partial_{\xi_\mu}\partial_{\xi_\nu}(\sigma_2)\partial_{x_\mu}\partial_{x_\nu}(\sigma_2^{-1})-\frac{1}{2}\partial_{\xi_\mu}\partial_{\xi_\nu}(\sigma_2^{-1})\partial_{x_\mu}\partial_{x_\nu}(\sigma_2^{-1}).
\end{align}
Using
\begin{align}\label{bbb}
&\partial_{\xi_\mu}\partial_{\xi_\nu}\sigma^{-1}_{2}(x_0)=-\sigma^{-2}_{2}\partial_{\xi_\mu}\partial_{\xi_\nu}\sigma_2+2\sigma_2^{-3}\partial_{\xi_\mu}\sigma_{2}\partial_{\xi_\nu}\sigma_2;\nonumber\\
&\partial_{x_\mu}\partial_{x_\nu}\sigma_2^{-1}(x_0)=|\xi|^{-2}\partial_{x_\mu}\partial_{x_\nu}[(fh)^{-2}]-(fh)^{-2}|\xi|^{-4}\sum_{jl}\partial_{x_\mu}\partial_{x_\nu}g^{jl}(x_0)\xi_j\xi_l,
\end{align}
then substitute the result in (\ref{bbb}) into the (\ref{ccc}), Lemma \ref{999} holds.
\end{proof}
By Lemma \ref{lemaa3}-Lemma \ref{999}, then the following equality holds.
\begin{align}
\sigma_{-6}^{Q^{-2}}(x_0,\xi)
&=-\frac{1}{2}(fh)^{-4}|\xi|^{-6}s+2(fh)^{-4}|\xi|^{-8}R_{\alpha a\alpha\mu}(x_0)\xi_a\xi_\mu-12(fh)^{-6}f^2|\xi|^{-8}\sum_{jl}\partial_{x_j}(h)\partial_{x_l}(h)\xi_j\xi_l\nonumber\\
&+44(fh)^{-6}f|\xi|^{-8}\sum_{jl}\partial_{x_j}(h)\partial_{x_l}(fh)\xi_j\xi_l-10(fh)^{-6}f|\xi|^{-6}\sum_{j}\partial_{x_j}(h)\partial_{x_j}(fh)\nonumber\\
&-12(fh)^{-2}|\xi|^{-8}\sum_{jl}\partial_{x_j}[(fh)^{-3}f]\partial_{x_l}(h)\xi_j\xi_l-12(fh)^{-5}f|\xi|^{-8}\sum_{jl}\partial_{x_j}\partial_{x_l}(h)\xi_j\xi_l\nonumber\\
&+24(fh)^{-2}|\xi|^{-8}\sum_{jl}\partial_{x_j}[(fh)^{-3}]\partial_{x_l}(fh)\xi_j\xi_l+24(fh)^{-5}|\xi|^{-8}\sum_{jl}\partial_{x_j}\partial_{x_l}(fh)\xi_j\xi_l\nonumber\\
&+3(fh)^{-2}|\xi|^{-6}\sum_{j}\partial_{x_j}^2[(fh)^{-2}]+14(fh)^{-6}f|\xi|^{-8}\sum_{j}\partial_{x_j}(h)\xi_jc(d(hf))c(\xi)\nonumber\\
&-28(fh)^{-6}|\xi|^{-8}\sum_{j}\partial_{x_j}(fh)\xi_jc(d(hf))c(\xi)-4(fh)^{-6}|\xi|^{-8}c(d(hf))c(\xi)c(d(hf))c(\xi)\nonumber\\
&+6(fh)^{-6}|\xi|^{-6}\sum_\mu\partial_{x_\mu}(fh)c(d(hf))\partial_{\xi_\mu}[c(\xi)]+2(fh)^{-5}f|\xi|^{-6}\sum_{j}\partial_{x_j}^2(h)\nonumber\\
&-2(fh)^{-6}f|\xi|^{-6}c(d(hf))c(dh)+2(fh)^{-2}|\xi|^{-8}\sum_{jl}\partial_{x_j}\partial_{x_l}[(fh)^{-2}]\xi_j\xi_l\nonumber\\
&-42(fh)^{-6}|\xi|^{-8}\sum_{jl}\partial_{x_j}(fh)\partial_{x_l}(fh)\xi_j\xi_l+6(fh)^{-5}|\xi|^{-8}\sum_{j}c(d(hf))\partial_{x_j}[c(\xi)]\xi_j\nonumber\\
&+8(fh)^{-6}|\xi|^{-6}\sum_j\partial_{x_j}(fh)\partial_{x_j}(fh)+6(fh)^{-2}|\xi|^{-8}]\sum_{j}\partial_{x_j}[(fh)^{-3}]\xi_jc(d(hf))c(\xi).
\end{align}
 \section{A Kastler-Kalau-Walze type theorem on 6-dimensional manifolds without boundary}
\label{section:3}

In the section, we prove the Kastler-Kalau-Walze type theorem for the rescaled Dirac operator on 6-dimensional manifolds without boundary. Fistly, we review here technical tool of the computation, which are the integrals of polynomial functions over the unit spheres. By (32) in \cite{B1}, we define
\begin{align}
I_{S_n}^{\gamma_1\cdot\cdot\cdot\gamma_{2\bar{n}+2}}=\int_{|x|=1}d^nxx^{\gamma_1}\cdot\cdot\cdot x^{\gamma_{2\bar{n}+2}},
\end{align}
i.e. the monomial integrals over a unit sphere.
Then by Proposition A.2. in \cite{B1},  polynomial integrals over higher spheres in the $n$-dimesional case are given
\begin{align}
I_{S_n}^{\gamma_1\cdot\cdot\cdot\gamma_{2\bar{n}+2}}=\frac{1}{2\bar{n}+n}[\delta^{\gamma_1\gamma_2}I_{S_n}^{\gamma_3\cdot\cdot\cdot\gamma_{2\bar{n}+2}}+\cdot\cdot\cdot+\delta^{\gamma_1\gamma_{2\bar{n}+1}}I_{S_n}^{\gamma_2\cdot\cdot\cdot\gamma_{2\bar{n}+1}}],
\end{align}
where $S_n\equiv S^{n-1}$ in $\mathbb{R}^n$.

For $\bar{n}=0$, we have $I^0$=area$(S_n)$=$\frac{2\pi^{\frac{n}{2}}}{\Gamma(\frac{n}{2})}$, we immediately get
\begin{align}
I_{S_n}^{\gamma_1\gamma_2}&=\frac{1}{n}area(S_n)\delta^{\gamma_1\gamma_2};\nonumber\\
I_{S_n}^{\gamma_1\gamma_2\gamma_3\gamma_4}&=\frac{1}{n(n+2)}area(S_n)[\delta^{\gamma_1\gamma_2}\delta^{\gamma_3\gamma_4}+\delta^{\gamma_1\gamma_3}\delta^{\gamma_2\gamma_4}+\delta^{\gamma_1\gamma_4}\delta^{\gamma_2\gamma_3}];\nonumber\\
I_{S_n}^{\gamma_1\gamma_2\gamma_3\gamma_4\gamma_5\gamma_6}&=\frac{1}{n(n+2)(n+4)}area(S_n)[\delta^{\gamma_1\gamma_2}(\delta^{\gamma_3\gamma_4}\delta^{\gamma_5\gamma_6}+\delta^{\gamma_3\gamma_5}\delta^{\gamma_4\gamma_6}+\delta^{\gamma_3\gamma_6}\delta^{\gamma_4\gamma_5})\nonumber\\
&+\delta^{\gamma_1\gamma_3}(\delta^{\gamma_2\gamma_4}\delta^{\gamma_5\gamma_6}+\delta^{\gamma_2\gamma_5}\delta^{\gamma_4\gamma_6}+\delta^{\gamma_2\gamma_6}\delta^{\gamma_4\gamma_5})+\delta^{\gamma_1\gamma_4}(\delta^{\gamma_2\gamma_3}\delta^{\gamma_5\gamma_6}+\delta^{\gamma_2\gamma_5}\delta^{\gamma_3\gamma_6}+\delta^{\gamma_2\gamma_6}\delta^{\gamma_3\gamma_5})\nonumber\\
&+\delta^{\gamma_1\gamma_5}(\delta^{\gamma_2\gamma_3}\delta^{\gamma_4\gamma_6}+\delta^{\gamma_2\gamma_4}\delta^{\gamma_3\gamma_6}+\delta^{\gamma_2\gamma_6}\delta^{\gamma_3\gamma_4})+\delta^{\gamma_1\gamma_6}(\delta^{\gamma_2\gamma_3}\delta^{\gamma_4\gamma_5}+\delta^{\gamma_2\gamma_4}\delta^{\gamma_3\gamma_5}+\delta^{\gamma_2\gamma_5}\delta^{\gamma_3\gamma_4})].
\end{align}
Then, we calculate each term of  $\int_{|\xi|=1}{\rm tr}[\sigma_{-6}^{Q^{-2}}(x_0,\xi)]\sigma(\xi)$ separately.\\
$\mathbf{(1)}$
\begin{align*}
{\rm tr}\bigg(-\frac{1}{2}(fh)^{-4}|\xi|^{-6}s\bigg)|_{|\xi|=1}&=-\frac{1}{2}(fh)^{-4}s{\rm tr}[id],\nonumber\\
\end{align*}
then
\begin{align*}
\int_{|\xi|=1}{\rm tr}\bigg(-\frac{1}{2}(fh)^{-4}s{\rm tr}[id]\bigg)\sigma(\xi)&=-\frac{1}{2}(fh)^{-4}s{\rm tr}[id]area(S_6).
\end{align*}
$\mathbf{(2)}$
\begin{align*}
{\rm tr}\bigg(2(fh)^{-4}|\xi|^{-8}R_{\alpha a\alpha\mu}(x_0)\xi_a\xi_\mu\bigg)|_{|\xi|=1}&=2(fh)^{-4}R_{\alpha a\alpha\mu}(x_0)\xi_a\xi_\mu{\rm tr}[id],
\end{align*}
then
\begin{align*}
\int_{|\xi|=1}{\rm tr}\bigg(2(fh)^{-4}R_{\alpha a\alpha\mu}(x_0)\xi_a\xi_\mu{\rm tr}[id]\bigg)\sigma(\xi)&=2(fh)^{-4}R_{\alpha a\alpha\mu}(x_0)\int_{|\xi|=1}\xi_a\xi_\mu\sigma(\xi){\rm tr}[id]\nonumber\\
&=2(fh)^{-4}s\times\frac{1}{6}area(S_6){\rm tr}[id]\nonumber\\
&=\frac{1}{3}(fh)^{-4}s{\rm tr}[id]area(S_6).
\end{align*}
$\mathbf{(3)}$
\begin{align*}
{\rm tr}\bigg(-12(fh)^{-6}f^2|\xi|^{-8}\sum_{jl}\partial_{x_j}(h)\partial_{x_l}(h)\xi_j\xi_l\bigg)|_{|\xi|=1}&=12(fh)^{-6}f^2\sum_{jl}\partial_{x_j}(h)\partial_{x_l}(h)\xi_j\xi_l{\rm tr}[id],
\end{align*}
then
\begin{align*}
&\int_{|\xi|=1}{\rm tr}\bigg(-12(fh)^{-6}f^2\sum_{jl}\partial_{x_j}(h)\partial_{x_l}(h)\xi_j\xi_l{\rm tr}[id]\bigg)\sigma(\xi)\nonumber\\
&=-12(fh)^{-6}f^2\sum_{jl}\partial_{x_j}(h)\partial_{x_l}(h)\int_{|\xi|=1}\xi_j\xi_l\sigma(\xi){\rm tr}[id]\nonumber\\
&=-2(fh)^{-6}f^2\sum_{j}\partial_{x_j}(h)\partial_{x_j}(h){\rm tr}[id]area(S_6)\nonumber\\
&=-2(fh)^{-6}f^2g\Big(\nabla(h),\nabla(h)\Big){\rm tr}[id]area(S_6)\nonumber\\
&=-2(fh)^{-6}f^2|\nabla(h)|^2{\rm tr}[id]area(S_6),
\end{align*}
where $\nabla(h)$ denotes the gradient of $h$.\\
$\mathbf{(4)}$
Similar to $\mathbf{(3)}$, we have
\begin{align*}
\int_{|\xi|=1}{\rm tr}\bigg(44(fh)^{-6}f|\xi|^{-8}\sum_{jl}\partial_{x_j}(h)\partial_{x_l}(fh)\xi_j\xi_l\bigg)\sigma(\xi)&=\frac{22}{3}(fh)^{-6}fg\Big(\nabla(h),\nabla(fh)\Big){\rm tr}[id]area(S_6).
\end{align*}
$\mathbf{(5)}$
Similar to $\mathbf{(3)}$, we have
\begin{align*}
\int_{|\xi|=1}{\rm tr}\bigg(-10(fh)^{-6}f|\xi|^{-6}\sum_{j}\partial_{x_j}(h)\partial_{x_j}(fh)\bigg)\sigma(\xi)&=-10(fh)^{-6}fg\Big(\nabla(h),\nabla(fh)\Big){\rm tr}[id]area(S_6).
\end{align*}
$\mathbf{(6)}$
Similar to $\mathbf{(3)}$, we have
\begin{align*}
\int_{|\xi|=1}{\rm tr}\bigg(-12(fh)^{-2}|\xi|^{-8}\sum_{jl}\partial_{x_j}[(fh)^{-3}f]\partial_{x_l}(h)\xi_j\xi_l\bigg)\sigma(\xi)&=-2(fh)^{-2}g\Big(\nabla[(fh)^{-3}f],\nabla(h)\Big){\rm tr}[id]area(S_6).
\end{align*}
$\mathbf{(7)}$
\begin{align*}
{\rm tr}\bigg(-12(fh)^{-5}f|\xi|^{-8}\sum_{jl}\partial_{x_j}\partial_{x_l}(h)\xi_j\xi_l
\bigg)|_{|\xi|=1}&=-12(fh)^{-5}f\sum_{jl}\partial_{x_j}\partial_{x_l}(h)\xi_j\xi_l{\rm tr}[id],
\end{align*}
then
\begin{align*}
&\int_{|\xi|=1}{\rm tr}\bigg(-12(fh)^{-5}f\sum_{jl}\partial_{x_j}\partial_{x_l}(h)\xi_j\xi_l{\rm tr}[id]\bigg)\sigma(\xi)\nonumber\\
&=-12(fh)^{-5}f\sum_{jl}\partial_{x_j}\partial_{x_l}(h)\int_{|\xi|=1}\xi_j\xi_l\sigma(\xi){\rm tr}[id]\nonumber\\
&=-2(fh)^{-5}f\sum_{j}\partial_{x_j}^2(h){\rm tr}[id]area(S_6)\nonumber\\
&=-2(fh)^{-5}f\Delta(h){\rm tr}[id]area(S_6),
\end{align*}
where $\Delta(h)$ denotes a generalized laplacian of $h$.\\
$\mathbf{(8)}$
Similar to $\mathbf{(3)}$, we have
\begin{align*}
\int_{|\xi|=1}{\rm tr}\bigg(24(fh)^{-2}|\xi|^{-8}\sum_{jl}\partial_{x_j}[(fh)^{-3}]\partial_{x_l}(fh)\xi_j\xi_l\bigg)\sigma(\xi)&=4(fh)^{-2}fg\Big(\nabla[(fh)^{-3}],\nabla(h)\Big){\rm tr}[id]area(S_6).
\end{align*}
$\mathbf{(9)}$
Similar to $\mathbf{(7)}$, we have
\begin{align*}
&\int_{|\xi|=1}{\rm tr}\bigg(24(fh)^{-5}|\xi|^{-8}\sum_{jl}\partial_{x_j}\partial_{x_l}(fh)\xi_j\xi_l\bigg)\sigma(\xi)=4(fh)^{-5}\Delta(fh){\rm tr}[id]area(S_6).
\end{align*}
$\mathbf{(10)}$
Similar to $\mathbf{(7)}$, we have
\begin{align*}
&\int_{|\xi|=1}{\rm tr}\bigg(3(fh)^{-2}|\xi|^{-6}\sum_{jl}\partial_{x_j}^2[(fh)^{-2}]\bigg)\sigma(\xi)=3(fh)^{-2}\Delta[(fh)^{-2}]{\rm tr}[id]area(S_6).
\end{align*}
$\mathbf{(11)}$
By $c(\xi)=\sum_j\xi_jc(dx_j)$ and $c(d(hf))=c(\sum_j\partial_{x_j}(hf)dx_j)=\sum_j\partial_{x_j}(hf)c(dx_j)$, we have
\begin{align*}
{\rm tr}\bigg(14(fh)^{-6}f|\xi|^{-8}\sum_{j}\partial_{x_j}(h)\xi_jc(d(hf))c(\xi)
\bigg)|_{|\xi|=1}&=-14(fh)^{-6}f\sum_{jl}\partial_{x_j}(h)\partial_{x_l}(hf)\xi_j\xi_l{\rm tr}[id],
\end{align*}
then
\begin{align*}
&\int_{|\xi|=1}{\rm tr}\bigg(-14(fh)^{-6}f\sum_{jl}\partial_{x_j}(h)\partial_{x_l}(hf)\xi_j\xi_l{\rm tr}[id]\bigg)\sigma(\xi)\nonumber\\
&=-14(fh)^{-6}f\sum_{jl}\partial_{x_j}(h)\partial_{x_l}(hf)\int_{|\xi|=1}\xi_j\xi_l\sigma(\xi){\rm tr}[id]\nonumber\\
&=-\frac{7}{3}(fh)^{-6}f\sum_{jl}\partial_{x_j}(h)\partial_{x_l}(hf){\rm tr}[id]area(S_6)\nonumber\\
&=-\frac{7}{3}(fh)^{-6}fg\Big(\nabla(h),\nabla(hf)\Big){\rm tr}[id]area(S_6),
\end{align*}
$\mathbf{(12)}$
Similar to $\mathbf{(11)}$, we have
\begin{align*}
&\int_{|\xi|=1}{\rm tr}\bigg(-28(fh)^{-6}|\xi|^{-8}\sum_{j}\partial_{x_j}(fh)\xi_jc(d(hf))c(\xi)\bigg)\sigma(\xi)=\frac{14}{3}(fh)^{-6}|\nabla(fh)|^{2}{\rm tr}[id]area(S_6).
\end{align*}
$\mathbf{(13)}$
\begin{align*}
&{\rm tr}\bigg(-4(fh)^{-6}|\xi|^{-8}c(d(hf))c(\xi)c(d(hf))c(\xi)
\bigg)|_{|\xi|=1}\nonumber\\
&=-4(fh)^{-6}{\rm tr}[c(d(hf))c(\xi)c(d(hf))c(\xi)]\nonumber\\
&=-4(fh)^{-6}\sum_{ijkl}\partial_{x_i}(hf)\partial_{x_k}(hf)\xi_j\xi_l{\rm tr}[c(dx_i)c(dx_j)c(dx_k)c(dx_l)],
\end{align*}
where
\begin{align*}
{\rm tr}[c(dx_i)c(dx_j)c(dx_k)c(dx_l)]=-\delta_{ij}{\rm tr}[c(dx_k)c(dx_l)]+\delta_{ik}{\rm tr}[c(dx_j)c(dx_l)]-\delta_{il}{\rm tr}[c(dx_j)c(dx_k)]\bigg),
\end{align*}
then
\begin{align*}
-4(fh)^{-6}{\rm tr}[c(d(hf))c(\xi)c(d(hf))c(\xi)]&=\Big(\sum_{ik}\partial_{x_i}(hf)\partial_{x_k}(hf)\xi_i\xi_k-\sum_{ij}\partial_{x_i}(hf)\partial_{x_i}(hf)\xi_j\xi_j\nonumber\\
&+\sum_{ij}\partial_{x_i}(hf)\partial_{x_j}(hf)\xi_i\xi_j\Big){\rm tr}[id].
\end{align*}
Further,
\begin{align*}
&\int_{|\xi|=1}{\rm tr}\bigg(-4(fh)^{-6}{\rm tr}[c(d(fh))c(\xi)c(d(fh))c(\xi)]\bigg)\sigma(\xi)\nonumber\\
&=-4(fh)^{-6}\partial_{x_i}(fh)\partial_{x_i}(fh)\times\frac{1}{6}area(S_6){\rm tr}[id]\nonumber\\
&=-\frac{2}{3}(fh)^{-6}|\nabla(fh)|^2{\rm tr}[id]area(S_6).
\end{align*}
$\mathbf{(14)}$
By $\partial_{\xi_\mu}[c(\xi)]=c(dx_\mu)$, we have
\begin{align*}
&{\rm tr}\bigg(6(fh)^{-6}|\xi|^{-6}\sum_\mu\partial_{x_\mu}(fh)c(d(hf))\partial_{\xi_\mu}[c(\xi)]
\bigg)|_{|\xi|=1}\nonumber\\
&=6(fh)^{-6}\sum_\mu\partial_{x_\mu}(fh){\rm tr}[c(d(hf))\partial_{\xi_\mu}[c(\xi)]]\nonumber\\
&=-6(fh)^{-6}\sum_\mu\partial_{x_\mu}(fh)\partial_{x_\mu}(fh){\rm tr}[id],
\end{align*}
then,
\begin{align*}
&\int_{|\xi|=1}{\rm tr}\bigg(-6(fh)^{-6}\sum_\mu\partial_{x_\mu}(fh)\partial_{x_\mu}(fh){\rm tr}[id]\bigg)\sigma(\xi)\nonumber\\
&=-6(fh)^{-6}|\nabla(fh)|^2{\rm tr}[id]area(S_6).
\end{align*}
$\mathbf{(15)}$
Similar to $\mathbf{(7)}$, we have
\begin{align*}
&\int_{|\xi|=1}{\rm tr}\bigg(2(fh)^{-5}f|\xi|^{-6}\sum_{j}\partial_{x_j}^2(h)\bigg)\sigma(\xi)=2(fh)^{-5}f\Delta(h){\rm tr}[id]area(S_6).
\end{align*}
$\mathbf{(16)}$
Similar to $\mathbf{(11)}$, we have
\begin{align*}
&\int_{|\xi|=1}{\rm tr}\bigg(-2(fh)^{-6}f|\xi|^{-6}c(d(hf))c(dh)\bigg)\sigma(\xi)=2(fh)^{-6}fg\Big(\nabla(hf),\nabla(h)\Big){\rm tr}[id]area(S_6).
\end{align*}
$\mathbf{(17)}$
Similar to $\mathbf{(7)}$, we have
\begin{align*}
&\int_{|\xi|=1}{\rm tr}\bigg(-4(fh)^{-2}|\xi|^{-8}\sum_{jl}\partial_{x_j}\partial_{x_l}[(fh)^{-2}]\xi_j\xi_l\bigg)\sigma(\xi)=-\frac{2}{3}(fh)^{-2}\Delta[(fh)^{-2}]{\rm tr}[id]area(S_6).
\end{align*}
$\mathbf{(18)}$
Similar to $\mathbf{(3)}$, we have
\begin{align*}
&\int_{|\xi|=1}{\rm tr}\bigg(-42(fh)^{-6}|\xi|^{-8}\sum_{jl}\partial_{x_j}(fh)\partial_{x_l}(fh)\xi_j\xi_l\bigg)\sigma(\xi)=-7(fh)^{-6}|\nabla(fh)|^2{\rm tr}[id]area(S_6).
\end{align*}
$\mathbf{(19)}$
By $\partial_{x_j}[H_{t\alpha}](x_0)=\partial_{x_j}[c(e_\alpha)]=0,$ we have
\begin{align}
\partial_{x_j}[c(\xi)](x_0)&=\sum_t\xi_t\partial_{x_j}[c(dx_t)](x_0)\nonumber\\
&=\sum_t\xi_t\partial_{x_j}[\sum_\alpha(dx_t,e_\alpha)c(e_\alpha)](x_0)\nonumber\\
&=\sum_{t\alpha}\xi_t\partial_{x_j}[H_{t\alpha}c(e_\alpha)](x_0)\nonumber\\
&=0.
\end{align}
Then
\begin{align*}
&\int_{|\xi|=1}{\rm tr}\bigg(6(fh)^{-5}|\xi|^{-8}\sum_{j}c(d(fh))\partial_{x_j}[c(\xi)]\xi_j\bigg)\sigma(\xi)=0.
\end{align*}
$\mathbf{(20)}$
Similar to $\mathbf{(3)}$, we have
\begin{align*}
&\int_{|\xi|=1}{\rm tr}\bigg(8(fh)^{-6}|\xi|^{-6}\sum_j\partial_{x_j}(fh)\partial_{x_j}(fh)\bigg)\sigma(\xi)=8(fh)^{-6}|\nabla(fh)|^2{\rm tr}[id]area(S_6).
\end{align*}
$\mathbf{(21)}$
Similar to $\mathbf{(11)}$, we have
\begin{align*}
&\int_{|\xi|=1}{\rm tr}\bigg(6(fh)^{-2}|\xi|^{-8}]\sum_{j}\partial_{x_j}[(fh)^{-3}]\xi_jc(d(hf))c(\xi)\bigg)\sigma(\xi)=-(fh)^{-2}g\Big(\nabla[(fh)^{-3}],\nabla(hf)\Big){\rm tr}[id]area(S_6).
\end{align*}
Therefore, we get
\begin{align*}
&\int_{|\xi|=1}{\rm tr}[\sigma_{-6}^{Q^{-2}}(x_0,\xi)]\sigma(\xi)\nonumber\\
&=\bigg(-\frac{1}{6}(fh)^{-4}s-2(fh)^{-6}f^2|\nabla(h)|^2-3(fh)^{-6}fg\Big(\nabla(h),\nabla(fh)\Big)-2(fh)^{-2}g\Big(\nabla[(fh)^{-3}f],\nabla(h)\Big)\nonumber\\
&+4(fh)^{-2}fg\Big(\nabla[(fh)^{-3}],\nabla(h)\Big)+4(fh)^{-5}\Delta(fh)+3(fh)^{-2}\Delta[(fh)^{-2}]-(fh)^{-6}|\nabla(fh)|^{2}\nonumber\\
&-\frac{2}{3}(fh)^{-2}\Delta[(fh)^{-2}]-(fh)^{-2}g\Big(\nabla[(fh)^{-3}],\nabla(hf)\Big)\bigg){\rm tr}[id]area(S_6).
\end{align*}
When $n=6,~~m=3$, we have ${\rm tr}[id]=2^3=8,$ $area(S_6)=\pi^3,$ then we get the following theorem.
\begin{thm}\label{thmooo}
Let $M$ be a $6$-dimensional oriented
compact spin manifold without boundary, $Q=(fDh)^{2}$, then the Kastler-Kalau-Walze type theorem for the rescaled Dirac operator is given
 \begin{align}
&{\rm Wres}(Q^{-2})\nonumber\\
&=8\pi^3\int_M\bigg(-\frac{1}{6}(fh)^{-4}s-2(fh)^{-6}f^2|\nabla(h)|^2-3(fh)^{-6}fg\Big(\nabla(h),\nabla(fh)\Big)-2(fh)^{-2}g\Big(\nabla[(fh)^{-3}f],\nabla(h)\Big)\nonumber\\
&+4(fh)^{-2}fg\Big(\nabla[(fh)^{-3}],\nabla(h)\Big)+4(fh)^{-5}\Delta(fh)+3(fh)^{-2}\Delta[(fh)^{-2}]-(fh)^{-6}|\nabla(fh)|^{2}\nonumber\\
&-\frac{2}{3}(fh)^{-2}\Delta[(fh)^{-2}]-(fh)^{-2}g\Big(\nabla[(fh)^{-3}],\nabla(hf)\Big)\bigg)d{\rm Vol_M}.
 \end{align}
 \end{thm}
 \section{A Kastler-Kalau-Walze type theorem on 6-dimensional manifolds with boundary}
\label{section:4}
In this section, we prove the Kastler-Kalau-Walze type theorem for the rescaled Dirac operator on $6$-dimensional manifolds with boundary. We firstly recall some basic facts and formulas about Boutet de
Monvel's calculus and the definition of the noncommutative residue for manifolds with boundary which will be used in the following. For more details, see in Section 2 in \cite{Wa3}. We assume that the metric $g^{M}$ on $M$ has the following form near the boundary,
\begin{equation}\label{a1}
g^{M}=\frac{1}{h(x_{n})}g^{\partial M}+dx _{n}^{2},
\end{equation}
where $g^{\partial M}$ is the metric on $\partial M$ and $h(x_n)\in C^{\infty}([0, 1)):=\{\widehat{h}|_{[0,1)}|\widehat{h}\in C^{\infty}((-\varepsilon,1))\}$ for
some $\varepsilon>0$ and $h(x_n)$ satisfies $h(x_n)>0$, $h(0)=1$ where $x_n$ denotes the normal directional coordinate. Let $U\subset M$ be a collar neighborhood of $\partial M$ which is diffeomorphic with $\partial M\times [0,1)$. By the definition of $h(x_n)\in C^{\infty}([0,1))$
and $h(x_n)>0$, there exists $\widehat{h}\in C^{\infty}((-\varepsilon,1))$ such that $\widehat{h}|_{[0,1)}=h$ and $\widehat{h}>0$ for some
sufficiently small $\varepsilon>0$. Then there exists a metric $g'$ on $\widetilde{M}=M\bigcup_{\partial M}\partial M\times
(-\varepsilon,0]$ which has the form on $U\bigcup_{\partial M}\partial M\times (-\varepsilon,0 ]$
\begin{equation}\label{a2}
g'=\frac{1}{\widehat{h}(x_{n})}g^{\partial M}+dx _{n}^{2} ,
\end{equation}
such that $g'|_{M}=g$. We fix a metric $g'$ on the $\widetilde{M}$ such that $g'|_{M}=g$.

Let the Fourier transformation $F'$  be
\begin{equation*}
F':L^2({\bf R}_t)\rightarrow L^2({\bf R}_v);~F'(u)(v)=\int_\mathbb{R} e^{-ivt}u(t)dt
\end{equation*}
and let
let
\begin{equation*}
r^{+}:C^\infty ({\bf R})\rightarrow C^\infty (\widetilde{{\bf R}^+});~ f\rightarrow f|\widetilde{{\bf R}^+};~
\widetilde{{\bf R}^+}=\{x\geq0;x\in {\bf R}\}.
\end{equation*}
\indent We define $H^+=F'(\Phi(\widetilde{{\bf R}^+}));~ H^-_0=F'(\Phi(\widetilde{{\bf R}^-}))$ which satisfies
$H^+\bot H^-_0$, where $\Phi(\widetilde{{\bf R}^+}) =r^+\Phi({\bf R})$, $\Phi(\widetilde{{\bf R}^-}) =r^-\Phi({\bf R})$ and $\Phi({\bf R})$
denotes the Schwartz space. We have the following
 property: $h\in H^+~$ (resp. $H^-_0$) if and only if $h\in C^\infty({\bf R})$ which has an analytic extension to the lower (resp. upper) complex
half-plane $\{{\rm Im}\xi<0\}$ (resp. $\{{\rm Im}\xi>0\})$ such that for all nonnegative integer $l$,
 \begin{equation*}
\frac{d^{l}h}{d\xi^l}(\xi)\sim\sum^{\infty}_{k=1}\frac{d^l}{d\xi^l}(\frac{c_k}{\xi^k}),
\end{equation*}
as $|\xi|\rightarrow +\infty,{\rm Im}\xi\leq0$ (resp. ${\rm Im}\xi\geq0)$ and where $c_k\in\mathbb{C}$ are some constants.\\
 \indent Let $H'$ be the space of all polynomials and $H^-=H^-_0\bigoplus H';~H=H^+\bigoplus H^-.$ Denote by $\pi^+$ (resp. $\pi^-$) respectively the
 projection on $H^+$ (resp. $H^-$). Let $\widetilde H=\{$rational functions having no poles on the real axis$\}$. Then on $\tilde{H}$
 \begin{equation}\label{a3}
\pi^+h(\xi_0)=\frac{1}{2\pi i}\lim_{u\rightarrow 0^{-}}\int_{\Gamma^+}\frac{h(\xi)}{\xi_0+iu-\xi}d\xi,
\end{equation}
where $\Gamma^+$ is a Jordan close curve
included ${\rm Im}(\xi)>0$ surrounding all the singularities of $h$ in the upper half-plane and
$\xi_0\in {\bf R}$. In our computations, we only compute $\pi^+h$ for $h$ in $\widetilde{H}$. Similarly, define $\pi'$ on $\tilde{H}$,
\begin{equation}\label{a4}
\pi'h=\frac{1}{2\pi}\int_{\Gamma^+}h(\xi)d\xi.
\end{equation}
So $\pi'(H^-)=0$. For $h\in H\bigcap L^1({\bf R})$, $\pi'h=\frac{1}{2\pi}\int_{{\bf R}}h(v)dv$ and for $h\in H^+\bigcap L^1({\bf R})$, $\pi'h=0$.\\
\indent An operator of order $m\in {\bf Z}$ and type $d$ is a matrix\\
$$\widetilde{A}=\left(\begin{array}{lcr}
  \pi^+P+G  & K  \\
   T  &  \widetilde{S}
\end{array}\right):
\begin{array}{cc}
\   C^{\infty}(M,E_1)\\
 \   \bigoplus\\
 \   C^{\infty}(\partial{M},F_1)
\end{array}
\longrightarrow
\begin{array}{cc}
\   C^{\infty}(M,E_2)\\
\   \bigoplus\\
 \   C^{\infty}(\partial{M},F_2)
\end{array},
$$
where $M$ is a manifold with boundary $\partial M$ and
$E_1,E_2$~ (resp. $F_1,F_2$) are vector bundles over $M~$ (resp. $\partial M
$).~Here,~$P:C^{\infty}_0(\Omega,\overline {E_1})\rightarrow
C^{\infty}(\Omega,\overline {E_2})$ is a classical
pseudodifferential operator of order $m$ on $\Omega$, where
$\Omega$ is a collar neighborhood of $M$ and
$\overline{E_i}|M=E_i~(i=1,2)$. $P$ has an extension:
$~{\cal{E'}}(\Omega,\overline {E_1})\rightarrow
{\cal{D'}}(\Omega,\overline {E_2})$, where
${\cal{E'}}(\Omega,\overline {E_1})~({\cal{D'}}(\Omega,\overline
{E_2}))$ is the dual space of $C^{\infty}(\Omega,\overline
{E_1})~(C^{\infty}_0(\Omega,\overline {E_2}))$. Let
$e^+:C^{\infty}(M,{E_1})\rightarrow{\cal{E'}}(\Omega,\overline
{E_1})$ denote extension by zero from $M$ to $\Omega$ and
$r^+:{\cal{D'}}(\Omega,\overline{E_2})\rightarrow
{\cal{D'}}(\Omega, {E_2})$ denote the restriction from $\Omega$ to
$X$, then define
$$\pi^+P=r^+Pe^+:C^{\infty}(M,{E_1})\rightarrow {\cal{D'}}(\Omega,
{E_2}).$$ In addition, $P$ is supposed to have the
transmission property; this means that, for all $j,k,\alpha$, the
homogeneous component $p_j$ of order $j$ in the asymptotic
expansion of the
symbol $p$ of $P$ in local coordinates near the boundary satisfies:\\
$$\partial^k_{x_n}\partial^\alpha_{\xi'}p_j(x',0,0,+1)=
(-1)^{j-|\alpha|}\partial^k_{x_n}\partial^\alpha_{\xi'}p_j(x',0,0,-1),$$
then $\pi^+P:C^{\infty}(M,{E_1})\rightarrow C^{\infty}(M,{E_2})$
by [12]. Let $G$,$T$ be respectively the singular Green operator
and the trace operator of order $m$ and type $d$. Let $K$ be a
potential operator and $S$ be a classical pseudodifferential
operator of order $m$ along the boundary (For detailed definition,
see [11]). Denote by $B^{m,d}$ the collection of all operators of
order $m$
and type $d$,  and $\mathcal{B}$ is the union over all $m$ and $d$.\\
\indent Recall that $B^{m,d}$ is a Fr\'{e}chet space. The composition
of the above operator matrices yields a continuous map:
$B^{m,d}\times B^{m',d'}\rightarrow B^{m+m',{\rm max}\{
m'+d,d'\}}.$ Write $$\widetilde{A}=\left(\begin{array}{lcr}
 \pi^+P+G  & K \\
 T  &  \widetilde{S}
\end{array}\right)
\in B^{m,d},
 \widetilde{A}'=\left(\begin{array}{lcr}
\pi^+P'+G'  & K'  \\
 T'  &  \widetilde{S}'
\end{array} \right)
\in B^{m',d'}.$$\\
 The composition $\widetilde{A}\widetilde{A}'$ is obtained by
multiplication of the matrices(For more details see [12]). For
example $\pi^+P\circ G'$ and $G\circ G'$ are singular Green
operators of type $d'$ and
$$\pi^+P\circ\pi^+P'=\pi^+(PP')+L(P,P').$$
Here $PP'$ is the usual
composition of pseudodifferential operators and $L(P,P')$ called
leftover term is a singular Green operator of type $m'+d$. For our case, $P,P'$ are classical pseudo differential operators, in other words $\pi^+P\in \mathcal{B}^{\infty}$ and $\pi^+P'\in \mathcal{B}^{\infty}$ .\\
\indent Let $M$ be an $n$-dimensional compact oriented manifold with boundary $\partial M$.
Denote by $\mathcal{B}$ the Boutet de Monvel's algebra. We recall that the main theorem in \cite{FGLS,Wa3}.
\begin{thm}\label{th:32}{\rm\cite{FGLS}}{\bf(Fedosov-Golse-Leichtnam-Schrohe)}
 Let $M$ and $\partial M$ be connected, ${\rm dim}M=n\geq3$, and let $\widetilde{S}$ (resp. $\widetilde{S}'$) be the unit sphere about $\xi$ (resp. $\xi'$) and $\sigma(\xi)$ (resp. $\sigma(\xi')$) be the corresponding canonical
$(n-1)$ (resp. $(n-2)$) volume form.
 Set $\widetilde{A}=\left(\begin{array}{lcr}\pi^+P+G &   K \\
T &  \widetilde{S}    \end{array}\right)$ $\in \mathcal{B}$ , and denote by $p$, $b$ and $s$ the local symbols of $P,G$ and $\widetilde{S}$ respectively.
 Define:
 \begin{align}
{\rm{\widetilde{Wres}}}(\widetilde{A})&=\int_X\int_{\bf \widetilde{ S}}{\rm{tr}}_E\left[p_{-n}(x,\xi)\right]\sigma(\xi)dx \nonumber\\
&+2\pi\int_ {\partial X}\int_{\bf \widetilde{S}'}\left\{{\rm tr}_E\left[({\rm{tr}}b_{-n})(x',\xi')\right]+{\rm{tr}}
_F\left[s_{1-n}(x',\xi')\right]\right\}\sigma(\xi')dx',
\end{align}
where ${\rm{\widetilde{Wres}}}$ denotes the noncommutative residue of an operator in the Boutet de Monvel's algebra.\\
Then~~ a) ${\rm \widetilde{Wres}}([\widetilde{A},B])=0 $, for any
$\widetilde{A},B\in\mathcal{B}$;~~ b) It is the unique continuous trace on
$\mathcal{B}/\mathcal{B}^{-\infty}$.
\end{thm}
\begin{defn}\label{def1}{\rm\cite{Wa3} }
Lower dimensional volumes of spin manifolds with boundary are defined by
 \begin{equation}\label{a6}
{\rm Vol}^{(p_1,p_2)}_nM:= \widetilde{{\rm Wres}}[\pi^+D^{-p_1}\circ\pi^+D^{-p_2}],
\end{equation}
\end{defn}
 By \cite{Wa3}, we get
\begin{equation}\label{a7}
\widetilde{{\rm Wres}}[\pi^+{Q}^{-p_1}\circ\pi^+{Q}^{-p_2}]=\int_M\int_{|\xi'|=1}{\rm
trace}_{\wedge^*T^*M\bigotimes\mathbb{C}}[\sigma_{-n}(Q^{-p_1-p_2})]\sigma(\xi)dx+\int_{\partial M}\Phi,
\end{equation}
and
\begin{align}\label{a8}
\Phi&=\int_{|\xi'|=1}\int^{+\infty}_{-\infty}\sum^{\infty}_{j, k=0}\sum\frac{(-i)^{|\alpha|+j+k+1}}{\alpha!(j+k+1)!}
\times {\rm trace}_{\wedge^*T^*M\bigotimes\mathbb{C}}[\partial^j_{x_n}\partial^\alpha_{\xi'}\partial^k_{\xi_n}\sigma^+_{r}(Q^{-p_1})(x',0,\xi',\xi_n)\nonumber\\
&\times\partial^\alpha_{x'}\partial^{j+1}_{\xi_n}\partial^k_{x_n}\sigma_{l}(Q^{-p_2})(x',0,\xi',\xi_n)]d\xi_n\sigma(\xi')dx',
\end{align}
 where the sum is taken over $r+l-k-|\alpha|-j-1=-n,~~r\leq -p_1,~l\leq -p_2$.

 Since $[\sigma_{-n}(Q^{-p_1-p_2})]|_M$ has the same expression as $\sigma_{-n}(Q^{-p_1-p_2})$ in the case of manifolds without
boundary, so locally we can compute the first term by \cite{KW}, \cite{Ka}, \cite{Wa3}.

For any fixed point $x_0\in\partial M$, we choose the normal coordinates
$U$ of $x_0$ in $\partial M$ (not in $M$) and compute $\Phi(x_0)$ in the coordinates $\widetilde{U}=U\times [0,1)\subset M$ and the
metric $\frac{1}{h(x_n)}g^{\partial M}+dx_n^2.$ The dual metric of $g^{M}$ on $\widetilde{U}$ is ${h(x_n)}g^{\partial M}+dx_n^2.$  Write
$g^{M}_{ij}=g^{M}(\frac{\partial}{\partial x_i},\frac{\partial}{\partial x_j});~ g_{M}^{ij}=g^{M}(dx_i,dx_j)$, then

\begin{equation}\label{a9}
[g^{M}_{ij}]= \left[\begin{array}{lcr}
  \frac{1}{h(x_n)}[g_{ij}^{\partial M}]  & 0  \\
   0  &  1
\end{array}\right];~~~
[g_{M}^{ij}]= \left[\begin{array}{lcr}
  h(x_n)[g^{ij}_{\partial M}]  & 0  \\
   0  &  1
\end{array}\right],
\end{equation}
and
\begin{equation}\label{a10}
\partial_{x_s}g_{ij}^{\partial M}(x_0)=0, 1\leq i,j\leq n-1; ~~~g_{ij}^{M}(x_0)=\delta_{ij}.
\end{equation}
\indent From \cite{Wa3}, we can get the following three lemmas,
\begin{lem}{\rm \cite{Wa3}}\label{lem1}
With the metric $g^{M}$ on $M$ near the boundary
\begin{align}\label{a11}
\partial_{x_j}(|\xi|_{g^{M}}^2)(x_0)&=\left\{
       \begin{array}{c}
        0,  ~~~~~~~~~~ ~~~~~~~~~~ ~~~~~~~~~~~~~{\rm if }~j<n, \\[2pt]
       h'(0)|\xi'|^{2}_{g^{\partial M}},~~~~~~~~~~~~~~~~~~~~{\rm if }~j=n,
       \end{array}
    \right. \\
\partial_{x_j}[c(\xi)](x_0)&=\left\{
       \begin{array}{c}
      0,  ~~~~~~~~~~ ~~~~~~~~~~ ~~~~~~~~~~~~~{\rm if }~j<n,\\[2pt]
\partial_{x_n}[c(\xi')](x_{0}), ~~~~~~~~~~~~~~~~~{\rm if }~j=n,
       \end{array}
    \right.
\end{align}
where $\xi=\xi'+\xi_{n}dx_{n}$.
\end{lem}
\begin{lem}{\rm \cite{Wa3}}\label{lem2}With the metric $g^{M}$ on $M$ near the boundary
\begin{align}\label{a12}
\omega_{s,t}(e_i)(x_0)&=\left\{
       \begin{array}{c}
        \omega_{n,i}(e_i)(x_0)=\frac{1}{2}h'(0),  ~~~~~~~~~~ ~~~~~~~~~~~{\rm if }~s=n,t=i,i<n, \\[2pt]
       \omega_{i,n}(e_i)(x_0)=-\frac{1}{2}h'(0),~~~~~~~~~~~~~~~~~~~{\rm if }~s=i,t=n,i<n,\\[2pt]
    \omega_{s,t}(e_i)(x_0)=0,~~~~~~~~~~~~~~~~~~~~~~~~~~~other~cases,~~~~~~~~~
       \end{array}
    \right.
\end{align}
where $(\omega_{s,t})$ denotes the connection matrix of Levi-Civita connection $\nabla^L$.
\end{lem}
\begin{lem}{\rm \cite{Wa3}}\label{lem3}When $i<n,$ then
\begin{align}
\label{a13}
\Gamma_{st}^k(x_0)&=\left\{
       \begin{array}{c}
        \Gamma^n_{ii}(x_0)=\frac{1}{2}h'(0),~~~~~~~~~~ ~~~~~~~~~~~{\rm if }~s=t=i,k=n, \\[2pt]
        \Gamma^i_{ni}(x_0)=-\frac{1}{2}h'(0),~~~~~~~~~~~~~~~~~~~{\rm if }~s=n,t=i,k=i,\\[2pt]
        \Gamma^i_{in}(x_0)=-\frac{1}{2}h'(0),~~~~~~~~~~~~~~~~~~~{\rm if }~s=i,t=n,k=i,\\[2pt]
       \end{array}
    \right.
\end{align}
in other cases, $\Gamma_{st}^i(x_0)=0$.
\end{lem}
\begin{thm}\label{bthm2}
Let $M$ be a $6$-dimensional oriented
compact spin manifold with boundary $\partial M$, $Q=(fDh)^{2}$, then the Kastler-Kalau-Walze type theorem for the rescaled Dirac operator is given
\begin{align}\label{a20}
&\widetilde{{\rm Wres}}[\pi^+Q^{-1}\circ\pi^+Q^{-1}]\nonumber\\
&=8\pi^3\int_M\bigg(-\frac{1}{6}(fh)^{-4}s-2(fh)^{-6}f^2|\nabla(h)|^2-3(fh)^{-6}fg\Big(\nabla(h),\nabla(fh)\Big)-2(fh)^{-2}g\Big(\nabla[(fh)^{-3}f],\nabla(h)\Big)\nonumber\\
&+4(fh)^{-2}fg\Big(\nabla[(fh)^{-3}],\nabla(h)\Big)+4(fh)^{-5}\Delta(fh)+3(fh)^{-2}\Delta[(fh)^{-2}]-(fh)^{-6}|\nabla(fh)|^{2}\nonumber\\
&-\frac{2}{3}(fh)^{-2}\Delta[(fh)^{-2}]-(fh)^{-2}g\Big(\nabla[(fh)^{-3}],\nabla(hf)\Big)\bigg)d{\rm Vol_M}.
\end{align}
In particular, the boundary term vanishes.
\end{thm}
\begin{proof}
\indent By (\ref{a7}) and (\ref{a8}), we firstly compute
\begin{equation}\label{a14}
\widetilde{{\rm Wres}}[\pi^+Q^{-1}\circ\pi^+Q^{-1}]=\int_M\int_{|\xi'|=1}{\rm
trace}_{\wedge^*T^*M\bigotimes\mathbb{C}}[\sigma_{-6}(Q^{-2})]\sigma(\xi)dx+\int_{\partial M}\Phi,
\end{equation}
where
\begin{align}\label{a15}
\Phi &=\int_{|\xi'|=1}\int^{+\infty}_{-\infty}\sum^{\infty}_{j, k=0}\sum\frac{(-i)^{|\alpha|+j+k+1}}{\alpha!(j+k+1)!}
\times {\rm trace}_{\wedge^*T^*M\bigotimes\mathbb{C}}[\partial^j_{x_n}\partial^\alpha_{\xi'}\partial^k_{\xi_n}\sigma^+_{r}(Q^{-1})(x',0,\xi',\xi_n)
\nonumber\\
&\times\partial^\alpha_{x'}\partial^{j+1}_{\xi_n}\partial^k_{x_n}\sigma_{l}(Q^{-1})(x',0,\xi',\xi_n)]d\xi_n\sigma(\xi')dx',
\end{align}
and the sum is taken over $r+l-k-j-|\alpha|=-3,~~r\leq -1,l\leq-1$.\\

\indent By Theorem \ref{thmooo}, we can compute the interior of $\widetilde{{\rm Wres}}[\pi^+Q^{-1}\circ\pi^+Q^{-1}]$, that is
\begin{align}\label{a16}
&{\rm Wres}[\pi^+Q^{-1}\circ\pi^+Q^{-1}]\nonumber\\
&=\int_M\int_{|\xi'|=1}{\rm
trace}_{S(TM)}[\sigma_{-6}(Q^{-2})]\sigma(\xi)dx\nonumber\\
&=8\pi^3\int_M\bigg(-\frac{1}{6}(fh)^{-4}s-2(fh)^{-6}f^2|\nabla(h)|^2-3(fh)^{-6}fg\Big(\nabla(h),\nabla(fh)\Big)-2(fh)^{-2}g\Big(\nabla[(fh)^{-3}f],\nabla(h)\Big)\nonumber\\
&+4(fh)^{-2}fg\Big(\nabla[(fh)^{-3}],\nabla(h)\Big)+4(fh)^{-5}\Delta(fh)+3(fh)^{-2}\Delta[(fh)^{-2}]-(fh)^{-6}|\nabla(fh)|^{2}\nonumber\\
&-\frac{2}{3}(fh)^{-2}\Delta[(fh)^{-2}]-(fh)^{-2}g\Big(\nabla[(fh)^{-3}],\nabla(hf)\Big)\bigg)d{\rm Vol_M}.
\end{align}
Next we can compute $\Phi$. Since $n=6$, then ${\rm tr}_{S(TM)}[\texttt{id}]=8$.
Since the sum is taken over $r+l-k-j-|\alpha|-1=-6, \ r\leq-2,~~l\leq -2$, then we have the $\int_{\partial_{M}}\Phi$
is the sum of the following five cases:\\
\noindent  {\bf case a)~I)}~$r=-2,~l=-2,~j=k=0,~|\alpha|=1$.\\
By (\ref{a15}), we get
 \begin{align}
{\rm \Phi_1}=-\int_{|\xi'|=1}\int^{+\infty}_{-\infty}\sum_{|\alpha|=1}{\rm trace}
\Big[\partial^{\alpha}_{\xi'}\pi^{+}_{\xi_{n}}\sigma_{-2}(Q^{-1})
      \times\partial^{\alpha}_{x'}\partial_{\xi_{n}}\sigma_{-2}(Q^{-1})\Big](x_0)d\xi_n\sigma(\xi')dx'.
\end{align}
By Lemma \ref{qwe}, for $i<n$, we have
 \begin{align}
 \partial_{x_{i}}\sigma_{-2}(Q^{-1})(x_0)&=
      \partial_{x_{i}}\Big((fh)^{-2}|\xi|^{-2}\Big)(x_{0})\nonumber\\
&=\partial_{x_{i}}[(fh)^{-2}]|\xi|^{-2}+(fh)^{-2}\partial_{x_{i}}[|\xi|^{-2}](x_0)\nonumber\\
&=\partial_{x_{i}}[(fh)^{-2}]|\xi|^{-2};
\end{align}
\begin{align}
\partial_{x_i}\partial_{\xi_{n}}\sigma_{-2}(Q^{-1})&=\partial_{\xi_{n}}\partial_{x_i}\sigma_{-2}(Q^{-1})\nonumber\\
&=-\frac{2\partial_{x_i}[(fh)^{-2}]\xi_n}{(1+\xi_n^2)^2},
\end{align}
and
\begin{align}
\partial_{\xi_i}\pi^{+}_{\xi_{n}}\sigma_{-2}(Q^{-1})&=\pi^{+}_{\xi_{n}}\partial_{\xi_i}\sigma_{-2}(Q^{-1})\nonumber\\
&=-\frac{(fh)^{-2}\xi^a(2+i\xi_n)}{4(\xi_n-i)^2}.
\end{align}
Then 
\begin{align}
\sum_{|\alpha|=1}{\rm trace}
\Big[\partial^{\alpha}_{\xi'}\pi^{+}_{\xi_{n}}\sigma_{-2}(Q^{-1})
      \times\partial^{\alpha}_{x'}\partial_{\xi_{n}}\sigma_{-2}(Q^{-1})\Big](x_0)=\frac{(fh)^{-2}\partial_{x_{i}}[(fh)^{-2}](2+i\xi_n)\xi_n\xi^a}{2(\xi_n-i)^4(\xi_n+i)^2}.
\end{align}
Note $i<n,~\int_{|\xi'|=1}\xi_{i_{1}}\xi_{i_{2}}\cdots\xi_{i_{2d+1}}\sigma(\xi')=0$,
so we have 
\begin{align}
{\rm \Phi_1}=0.
\end{align}
\noindent  {\bf case a)~II)}~$r=-2,~l=-2,~~|\alpha|=k=0,~~j=1$.\\
By (\ref{a15}), we have
  \begin{align}
{\rm \Phi_2}=-\frac{1}{2}\int_{|\xi'|=1}\int^{+\infty}_{-\infty} {\rm
trace} \Big[\partial_{x_{n}}\pi^{+}_{\xi_{n}}\sigma_{-2}(Q^{-1})
      \times\partial^{2}_{\xi_{n}}\sigma_{-2}(Q^{-1})\Big](x_0)d\xi_n\sigma(\xi')dx'.
\end{align}
By Lemma \ref{qwe} and direct calculations, we have
\begin{align}\label{b3}
\partial^{2}_{\xi_{n}}\sigma_{-2}(Q^{-1})=\partial^{2}_{\xi_{n}}[(fh)^{-2}|\xi|^{-2}]=\frac{(fh)^{-2}(6\xi^{2}_{n}-2)}{(1+\xi_{n}^{2})^{3}}.
\end{align}
\begin{align}\label{po}
\partial_{x_n}[\sigma_{-2}(Q^{-1})]=\frac{\partial_{x_n}[(fh)^{-2}]}{1+\xi_{n}^{2}}+\frac{(fh)^{-2}h'(0)}{(1+\xi_{n}^{2})^2}.
\end{align}
Then
\begin{align}\label{b77}
\partial_{x_{n}}\pi^{+}_{\xi_{n}}\sigma_{-2}(Q^{-1})&=\pi^{+}_{\xi_{n}}\partial_{x_{n}}\sigma_{-2}(Q^{-1})\nonumber\\
&=\pi^{+}_{\xi_{n}}\Big(\frac{\partial_{x_n}[(fh)^{-2}]}{1+\xi_{n}^{2}}+\frac{(fh)^{-2}h'(0)}{(1+\xi_{n}^{2})^2}\Big)\nonumber\\
&=\frac{-i\partial_{x_n}[(fh)^{-2}]}{2(\xi_{n}-i)}+\frac{(fh)^{-2}h'(0)(2+i\xi_n)}{4(\xi_{n}-i)^2}
\end{align}
By (\ref{b3}) and (\ref{b77}), we get
\begin{align}
&{\rm
trace} \Big[\partial_{x_{n}}\pi^{+}_{\xi_{n}}\sigma_{-1}(Q^{-1})
      \times\partial^{2}_{\xi_{n}}\sigma_{-2}(Q^{-1})\Big](x_0)\nonumber\\
&=-8i\partial_{x_n}[(fh)^{-2}](fh)^{-2}\frac{3\xi_{n}^2-1}{(\xi_{n}-i)^{4}(\xi_{n}+i)^{3}}+4h'(0)(fh)^{-4}\frac{(2+i\xi_n)(3\xi_{n}^2-1)}{(\xi_{n}-i)^{5}(\xi_{n}+i)^{3}}.
\end{align}
Then we obtain
\begin{align}
{\rm\Phi_2}&=-\frac{1}{2}\int_{|\xi'|=1}\int^{+\infty}_{-\infty} -8i\partial_{x_n}[(fh)^{-2}](fh)^{-2}\frac{3\xi_{n}^2-1}{(\xi_{n}-i)^{4}(\xi_{n}+i)^{3}}d\xi_n\sigma(\xi')dx'\nonumber\\
&-\frac{1}{2}\int_{|\xi'|=1}\int^{+\infty}_{-\infty} 4h'(0)(fh)^{-4}\frac{(2+i\xi_n)(3\xi_{n}^2-1)}{(\xi_{n}-i)^{5}(\xi_{n}+i)^{3}}d\xi_n\sigma(\xi')dx'\nonumber\\ &=8i\partial_{x_n}[(fh)^{-2}](fh)^{-2}\Omega_{4}\int_{\Gamma^{+}}\frac{3\xi_{n}^2-1}{(\xi_{n}-i)^{4}(\xi_{n}+i)^{3}}d\xi_{n}dx'\nonumber\\
&-h'(0)(fh)^{-4}\Omega_{4}\int_{\Gamma^{+}}\frac{(2+i\xi_n)(3\xi_{n}^2-1)}{(\xi_{n}-i)^{5}(\xi_{n}+i)^{3}}d\xi_{n}dx'\nonumber\\
     &=8i\partial_{x_n}[(fh)^{-2}](fh)^{-2}\Omega_{4}\frac{2\pi i}{3!}\bigg[\frac{3\xi_{n}^2-1}{(\xi_{n}+i)^{3}}\bigg]^{(3)}\bigg|_{\xi_{n}=i}dx'\nonumber\\
     &-h'(0)(fh)^{-4}\Omega_{4}\frac{2\pi i}{4!}\bigg[\frac{(2+i\xi_n)(3\xi_{n}^2-1)}{(\xi_{n}+i)^{4}}\bigg]^{(4)}\bigg|_{\xi_{n}=i}dx'\nonumber\\
     &=\Big(\frac{1}{2}(fh)^{-2}\partial_{x_n}[(fh)^{-2}]-\frac{5}{8}(fh)^{-4}h'(0)\Big)\pi \Omega_{4}dx',
\end{align}
where ${\rm \Omega_{4}}$ is the canonical volume of $S^{4}.$\\
\noindent  {\bf case a)~III)}~$r=-2,~~l=-2,~~|\alpha|=j=0,~~k=1$.\\
By (\ref{a15}), we have
 \begin{align}
{\rm \Phi_3}&=-\frac{1}{2}\int_{|\xi'|=1}\int^{+\infty}_{-\infty}{\rm trace} \Big[\partial_{\xi_{n}}\pi^{+}_{\xi_{n}}\sigma_{-2}(Q^{-1})
      \times\partial_{\xi_{n}}\partial_{x_{n}}\sigma_{-2}(Q^{-1})\Big](x_0)d\xi_n\sigma(\xi')dx'\nonumber\\
      &=\frac{1}{2}\int_{|\xi'|=1}\int^{+\infty}_{-\infty}{\rm trace} \Big[\partial_{\xi_{n}}^2\pi^{+}_{\xi_{n}}\sigma_{-2}(Q^{-1})
      \times\partial_{x_{n}}\sigma_{-2}(Q^{-1})\Big](x_0)d\xi_n\sigma(\xi')dx'.
\end{align}
By Lemma \ref{qwe} and direct calculations, we have
\begin{align}\label{b4}
\partial_{\xi_{n}}^2\pi^{+}_{\xi_{n}}\sigma_{-2}(Q^{-1})&=-(fh)^{-2}\frac{i}{(\xi_{n}-i)^{3}}.
\end{align}
Combining (\ref{po}) and (\ref{b4}), we have
\begin{align}
&{\rm trace} \Big[\partial_{\xi_{n}}^2\pi^{+}_{\xi_{n}}\sigma_{-2}(Q^{-1})
      \times\partial_{x_{n}}\sigma_{-2}(Q^{-1})\Big](x_{0})|_{|\xi'|=1}\nonumber\\
&=-8i\partial_{x_n}[(fh)^{-2}](fh)^{-2}\frac{1}{(\xi_{n}-i)^{4}(\xi_{n}+i)}+8h'(0)(fh)^{-4}\frac{1}{(\xi_{n}-i)^{5}(\xi_{n}+i)^{2}}.
\end{align}
Then
\begin{align}
{\rm \Phi_3}&=-\frac{1}{2}\int_{|\xi'|=1}\int^{+\infty}_{-\infty} -8i\partial_{x_n}[(fh)^{-2}](fh)^{-2}\frac{1}{(\xi_{n}-i)^{4}(\xi_{n}+i)}d\xi_n\sigma(\xi')dx'\nonumber\\
&-\frac{1}{2}\int_{|\xi'|=1}\int^{+\infty}_{-\infty}8h'(0)(fh)^{-4}\frac{1}{(\xi_{n}-i)^{5}(\xi_{n}+i)^{2}}d\xi_n\sigma(\xi')dx'\nonumber\\ &=4i\partial_{x_n}[(fh)^{-2}](fh)^{-2}\Omega_{4}\int_{\Gamma^{+}}\frac{1}{(\xi_{n}-i)^{4}(\xi_{n}+i)}d\xi_{n}dx'\nonumber\\
&-4h'(0)(fh)^{-4}\Omega_{4}\int_{\Gamma^{+}}\frac{1}{(\xi_{n}-i)^{5}(\xi_{n}+i)^{2}}d\xi_{n}dx'\nonumber\\
     &=4i\partial_{x_n}[(fh)^{-2}](fh)^{-2}\Omega_{4}\frac{2\pi i}{3!}\bigg[\frac{1}{(\xi_{n}+i)}\bigg]^{(3)}\bigg|_{\xi_{n}=i}dx'\nonumber\\
     &-4h'(0)(fh)^{-4}\Omega_{4}\frac{2\pi i}{4!}\bigg[\frac{1}{(\xi_{n}+i)^{2}}\bigg]^{(4)}\bigg|_{\xi_{n}=i}dx'\nonumber\\
     &=\Big(-\frac{1}{2}(fh)^{-2}\partial_{x_n}[(fh)^{-2}]+\frac{5}{8}(fh)^{-4}h'(0)\Big)\pi \Omega_{4}dx'.
\end{align}
\noindent  {\bf case b)}~$r=-2,~~l=-3,~~|\alpha|=j=k=0$.\\
By (\ref{a15}), we have
 \begin{align}
{\rm \Phi_4}&=-i\int_{|\xi'|=1}\int^{+\infty}_{-\infty}{\rm trace} \Big[\pi^{+}_{\xi_{n}}\sigma_{-2}(Q^{-1})
      \times\partial_{\xi_{n}}\sigma_{-3}(Q^{-1})\Big](x_0)d\xi_n\sigma(\xi')dx'\nonumber\\
&=i\int_{|\xi'|=1}\int^{+\infty}_{-\infty}{\rm trace}\Big[\partial_{\xi_{n}}\pi^{+}_{\xi_{n}}\sigma_{-2}(Q^{-1})
      \times\sigma_{-3}(Q^{-1})\Big](x_0)d\xi_n\sigma(\xi')dx'.
\end{align}
In the normal coordinate, $g^{ij}(x_{0})=\delta^{j}_{i}$ and $\partial_{x_{j}}(g^{\alpha\beta})(x_{0})=0$, if $j<n$; $\partial_{x_{j}}(g^{\alpha\beta})(x_{0})=h'(0)\delta^{\alpha}_{\beta}$, if $j=n$.
So by Lemma A.2 in \cite{Wa3}, we have $\Gamma^{n}(x_{0})=\frac{5}{2}h'(0)$ and $\Gamma^{k}(x_{0})=0$ for $k<n$. By the definition of $\sigma^{k}$ and Lemma 2.3 in \cite{Wa3}, we have $\sigma^{n}(x_{0})=0$ and $\sigma^{k}=\frac{1}{4}h'(0)c(e_{k})c(e_{n})$ for $k<n$. By Lemma \ref{qwe}, we obtain
\begin{align}\label{kkkk}
\sigma_{-3}(Q^{-1})(x_{0})|_{|\xi'|=1}&=(fh)^{-2}\Big[-\frac{i}{(1+\xi_n^2)^2}\Big(-\frac{1}{2}h'(0)\sum_{k<n}\xi_kc(e_{k})c(e_{n})+\frac{5}{2}h'(0)\xi_n\Big)-\frac{2ih'(0)\xi_n}{(1+\xi_n^2)^3}\Big]\nonumber\\
&+2i(fh)^{-3}f|\xi|^{-4}\sum_{j=1}^n\partial_{x_j}(h)\xi_j-i(fh)^{-3}|\xi|^{-4}c(d(hf))c(\xi)\nonumber\\
&-4i(fh)^{-3}|\xi|^{-4}\xi_\mu\partial_{x_\mu}(fh)
\end{align}
and
\begin{align}\label{b8}
\partial_{\xi_{n}}\pi^{+}_{\xi_{n}}\sigma_{-2}(Q^{-1})=(fh)^{-2}\frac{i}{2(\xi_n-i)^2}.
\end{align}
We note that $i<n,~\int_{|\xi'|=1}\xi_{i_{1}}\xi_{i_{2}}\cdots\xi_{i_{2d+1}}\sigma(\xi')=0$,
so we omit some items that have no contribution for computing {\rm case~b)}.
Then
\begin{align}\label{llop}
{\rm trace}\Big[\partial_{\xi_{n}}\pi^{+}_{\xi_{n}}\sigma_{-2}(Q^{-1})
      \times\sigma_{-3}(Q^{-1})\Big](x_0)&=\frac{(fh)^{-4}h'(0)}{(\xi_n-i)^2}\Big(\frac{10\xi_n}{(1+\xi_n^2)^2}+\frac{8\xi_n}{(1+\xi_n^2)^3}\Big).
\end{align}
Therefore, we have
\begin{align}
{\rm\Phi_4}&=
 i\int_{|\xi'|=1}\int^{+\infty}_{-\infty}\frac{(fh)^{-4}h'(0)}{(\xi_n-i)^2}\Big(\frac{10\xi_n}{(1+\xi_n^2)^2}+\frac{8\xi_n}{(1+\xi_n^2)^3}\Big)d\xi_n\sigma(\xi')dx'\nonumber\\
&=10i(fh)^{-4}\Omega_{4}\int_{\Gamma^{+}}\frac{\xi_n}{(\xi_{n}-i)^{4}(\xi_{n}+i)^2}d\xi_{n}dx'\nonumber\\
&+8ih'(0)(fh)^{-4}\Omega_{4}\int_{\Gamma^{+}}\frac{\xi_n}{(\xi_{n}-i)^{5}(\xi_{n}+i)^{3}}d\xi_{n}dx'\nonumber\\
     &=10i(fh)^{-4}\Omega_{4}\frac{2\pi i}{3!}\bigg[\frac{\xi_n}{(\xi_{n}+i)^2}\bigg]^{(3)}\bigg|_{\xi_{n}=i}dx'\nonumber\\
     &+8ih'(0)(fh)^{-4}\Omega_{4}\frac{2\pi i}{4!}\bigg[\frac{\xi_n}{(\xi_{n}+i)^{3}}\bigg]^{(4)}\bigg|_{\xi_{n}=i}dx'\nonumber\\
     &=-\frac{15}{8}(fh)^{-4}h'(0)\pi \Omega_{4}dx'.
\end{align}
\noindent {\bf  case c)}~$r=-3,~~l=-2,~~|\alpha|=j=k=0$.\\
By (\ref{a15}), we have
\begin{align}
{\rm\Phi_5}&=-i\int_{|\xi'|=1}\int^{+\infty}_{-\infty}{\rm trace} \Big[\pi^{+}_{\xi_{n}}\sigma_{-3}(Q^{-1})
      \times\partial_{\xi_{n}}\sigma_{-2}(Q^{-1})\Big](x_0)d\xi_n\sigma(\xi')dx'\nonumber\\
      &=-i\int_{|\xi'|=1}\int^{+\infty}_{-\infty}{\rm trace} \Big[\partial_{\xi_{n}}\sigma_{-2}(Q^{-1})
      \times\sigma_{-3}(Q^{-1})\Big](x_0)d\xi_n\sigma(\xi')dx'\nonumber\\
      &-i\int_{|\xi'|=1}\int^{+\infty}_{-\infty}{\rm trace} \Big[\pi^{+}_{\xi_{n}}\sigma_{-2}(Q^{-1})
      \times\partial_{\xi_{n}}\sigma_{-3}(Q^{-2})\Big](x_0)d\xi_n\sigma(\xi')dx'\nonumber\\
      &=\Phi_4-i\int_{|\xi'|=1}\int^{+\infty}_{-\infty}{\rm trace} \Big[\pi^{+}_{\xi_{n}}\sigma_{-2}(Q^{-1})
      \times\partial_{\xi_{n}}\sigma_{-3}(Q^{-2})\Big](x_0)d\xi_n\sigma(\xi')dx'.
\end{align}
By Lemma \ref{qwe}, we have
\begin{align}
\partial_{\xi_{n}}\sigma_{-2}(Q^{-1})&=-2(fh)^{-2}\frac{\xi_n}{(1+\xi_{n}^2)^{2}}.
\end{align}
We note that $i<n,~\int_{|\xi'|=1}\xi_{i_{1}}\xi_{i_{2}}\cdots\xi_{i_{2d+1}}\sigma(\xi')=0$,
so we omit some items that have no contribution for computing {\rm case~c)}.\\
Then by (\ref{kkkk}), we have
\begin{align}\label{B58}
{\rm trace} \Big[\pi^{+}_{\xi_{n}}\sigma_{-2}(Q^{-1})
      \times\partial_{\xi_{n}}\sigma_{-3}(Q^{-2})\Big](x_0)&=\frac{8(fh)^{-4}h'(0)(5i\xi_n^4+9i\xi_n^2)}{(\xi_n+i)^5(\xi_n-i)^5}.
\end{align}
Then we obtain
\begin{align}
{\rm \Phi_5}&=\Phi_4
 -i h'(0)\int_{|\xi'|=1}\int^{+\infty}_{-\infty}
 \frac{8(fh)^{-4}h'(0)(5i\xi_n^4+9i\xi_n^2)}{(\xi_n+i)^5(\xi_n-i)^5}d\xi_n\sigma(\xi')dx' \nonumber\\
&=\Phi_4+8(fh)^{-4}h'(0)\Omega_4\int_{\Gamma^{+}}\frac{5\xi_n^4+9\xi_n^2}{(\xi_{n}-i)^{5}(\xi_{n}+i)^{5}}d\xi_{n}dx'\nonumber\\
     &=\Phi_4+8(fh)^{-4}h'(0)\Omega_4\frac{2\pi i}{4!}\bigg[\frac{5\xi_n^4+9\xi_n^2}{(\xi_{n}+i)^{5}}\bigg]^{(4)}\bigg|_{\xi_{n}=i}dx'\nonumber\\
     &=\Phi_4+\frac{15}{4}(fh)^{-4}h'(0)\pi \Omega_{4}dx'\nonumber\\
     &=\frac{15}{8}(fh)^{-4}h'(0)\pi \Omega_{4}dx'.
\end{align}
Now $\Phi$ is the sum of the cases (a), (b) and (c), then
\begin{equation}\label{ppp}
\Phi=\sum_{i=1}^5\Phi_i=0.
\end{equation}
Then, by (\ref{ppp}), we obtain Theorem \ref{bthm2}.
\end{proof}

\section*{ Declarations}
\textbf{Ethics approval and consent to participate:} Not applicable.

\textbf{Consent for publication:} Not applicable.

\textbf{Availability of data and materials:} The authors confrm that the data supporting the fndings of this study are available within the article.

\textbf{Competing interests:} The authors declare no competing interests.

\textbf{Funding:} This research was funded by National Natural Science Foundation of China: No.11771070, 2023-BSBA-118 and N2405015..

\textbf{Author Contributions:} All authors contributed to the study conception and design. Material preparation,
data collection and analysis were performed by TW and YW. The frst draft of the manuscript was written
by TW and all authors commented on previous versions of the manuscript. All authors read and approved
the final manuscript.
\section*{Acknowledgements}
This work was supported by NSFC No.11771070, 2023-BSBA-118 and N2405015. The authors thank the referee for his (or her) careful reading and helpful comments.

\end{document}